\documentclass{amsart}
\usepackage {amssymb}
\usepackage {amsmath}
\usepackage{amsthm}
\usepackage{graphicx}
\usepackage {amscd}
\usepackage {epic}
\usepackage[alphabetic]{amsrefs}
\usepackage[colorlinks, citecolor=blue]{hyperref}

\newcommand{\mX}{{\mathcal{X}}}
\newcommand{\mY}{{\mathcal{Y}}}
\newcommand{\mZ}{{\mathcal{Z}}}

\newcommand{\Link}{{\rm Link}}

\newtheorem{thm}{Theorem}[section]

\newtheorem{lem}[thm]{Lemma}
\newtheorem{cor}[thm]{Corollary}

\newtheorem{prop}[thm]{Proposition}

\newtheorem{conj}[thm]{Conjecture}

\theoremstyle{definition}

\newtheorem{rem}[thm]{Remark}

\newtheorem*{ack}{Acknowledgments}      

\newtheorem{defn-thm}[thm]{Definition--Theorem}  
\newtheorem{defn-prop}[thm]{Definition--Proposition}  
\newtheorem{defn-lem}[thm]{Definition--Lemma}  
\newtheorem{convention}[thm]{Convention}  

\theoremstyle{remark}

\input{diagrams.sty}

\begin{document}
\title{Finiteness of fundamental groups}

\author{Zhiyu Tian}
\address{CNRS, Institut Fourier UMR 5582, 100 Rue des Math\'ematiques BP74, 38402 Saint-Martin d'H\`eres Cedex, France}
\email{zhiyu.tian@ujf-grenoble.fr}

\author    {Chenyang Xu}

\address   {Beijing International Center for Mathematical Research,
       Beijing 100871, China}
\email     {cyxu@math.pku.edu.cn}
\begin{abstract} We show that the finiteness of the fundamental groups of the smooth locus of {\it lower} dimensional log Fano pairs would imply the finiteness of the local fundamental group of klt singularities. As an application, we  verify that the local fundamental group  of a three dimensional klt singularity and the fundamental group of the smooth locus of a  three dimensional Fano  variety with canonical singularities are always finite.
\end{abstract}

\date{\today}

\maketitle{}
\tableofcontents
\section{Introduction}We work over the field $\mathbb{C}$ of complex numbers. In this note, we aim to study the local fundamental ${\pi}_1({\rm Link}(x\in X))$ of a klt singularity as well as the the fundamental group ${\pi}_1(X^{\rm sm})$ of the smooth open locus of a log Fano variety. These two fundamental groups are basic objects which appear naturally in many questions (e.g. \cite{Namikawa13, GKP13, KX15} etc.).

We have the following two conjectures.

\begin{conj}[Global Finiteness]\label{conj-gf}
Let $(X,D)$ be a log Fano pair, i.e., $(X,D)$ has klt singularities, and $-(K_X+D)$ is ample. Assume $D=\sum_i\frac{n_i-1}{n_i}D_i$.   Let $X^0\subset X$ be a open subset with ${\rm Codim}_{X}(X\setminus X^0)\ge 2$ such that $(X^0,{\rm Supp} (D) |_{X^0})$ is simple normal crossing.  Then the orbifold fundamental group  $\pi_1(X^0, D|_{X^0})$ is finite.
\end{conj}

\begin{conj}[Local Finiteness]\label{conj-lf}
Let $(X,D)$ be a  klt pair. Then $\pi_1(\Link(x\in X))$ is finite.
\end{conj}

As a folklore question, Conjecture \ref{conj-gf} has a long history (see e.g. \cite{Zhang95}, \cite[Question 0.11]{AIM07}). It has been confirmed for surfaces in \cite{FKL93, GZ94, GZ95, KM99}. On the other hand, Conjecture \ref{conj-lf}, as far as we know, was first formulated by Koll\'ar \cite[8.24]{Kollar13}.

A close relation of these two questions was first established in \cite{Xu14}, using the local-to-global induction. More precisely,  to study the local fundamental group of a klt singularity of dimension $n$,  we can use the minimal model program as established in \cite{BCHM10} to extract a Koll\'ar component which admits a natural structure of log Fano pairs and then use the information we know for an $(n-1)$-dimensional log Fano pairs. For the second step of studying the $n$-dimensional log Fano variety, since we know that the fundamental group of the compact total space is trivial, we can use the information of ($n$-dimensional) singularities to connect the fundamental of the smooth part and the fundamental group of the total space.

This strategy has been successfully applied in \cite{Xu14} to prove a weaker result which says that the algebraic fundamental groups, i.e., the pro-finite completions of the above two groups, are always finite (see also \cite{GKP13, KX15}).  In this note, we will first show that half of this strategy still works for topological fundamental group.

\begin{thm}\label{thm-lf}
Assume that Conjecture \ref{conj-gf} holds for dimension at most $n-1$, then Conjecture \ref{conj-lf} holds for dimension $n$.
\end{thm}

To prove this, the main ingredient is that we need to understand when we extract the Koll\'ar component, how the topology of its  punctured neighborhood $U^0$ relates to the Koll\'ar component itself. In general, we know that there is a topological contraction map from the punctured neighborhood to the Koll\'ar component. Let $V^0$ be the topological open set which contracts to the (orbifold) smooth part of the Koll\'ar component. Restricting on $V^0$, the fiber of this contraction is easy to understand. In fact, the standard calculation shows that this map is homotopic to an $S^1$-bundle. However, it is more tricky to understand the fiber over the singular points on the Koll\'ar component. For instance, the dimension of $U^0\setminus V^0$ is usually the same as $\dim (U^0)$.

Nevertheless, using the stratification theory for complex varieties (see e.g. \cite{GM88}), we can show that the fundamental group of $V^0$ indeed surjects to the fundamental group of $U^0$, hence we can conclude the finiteness of $\pi_1(U^0)$.

As a corollary, we immediately know the following result.
\begin{cor}\label{cor-3l}
 If $x\in X$ is a  three dimensional klt algebraic singularity, then
$\pi_1({\rm Link}(x\in X))$ is finite.
\end{cor}
We note that the finiteness of the pro-finite completion of the local fundamental group of a three dimensional klt singularity was known in \cite{SW94, Xu14}.

Assume that we know the finiteness of the local fundamental group for klt singularities with dimension at most $n$, then for any $n$-dimensional klt quasi-projective variety we can construct a complex space $Z$ with a properly discontinuously action by $G:=\pi_1(X^{\rm sm})$ such that $Z/G=X$ (see Proposition \ref{thm-cover}).

Of course we hope when the pair $(X,D)$ is log Fano, such an  covering $(Z,\Delta)$ is always of finite degree.
Unfortunately, we are lack of tools to show this in the general case. Nevertheless,  we confirm Conjecture \ref{conj-gf} for Fano threefolds with canonical singularities.

\begin{thm}\label{thm-3ter}
If $X$ is a Fano threefold with canonical singularities, then $\pi_1(X^{\rm sm})$ is finite.
\end{thm}

We prove this by combining the calculation in \cite{Reid87} with a generalization of the argument in \cite{Takayama00}: first we show that the $L^2$-index of the structure sheaf of $Z$ is non-zero, and then we obtain a contradiction by Gromov's technique if  the covering $Z$ is of infinite degree. We note that special cases of Theorem \ref{thm-3ter} with assumptions on the Fano index or Cartier index of $X$ of Theorem \ref{thm-3ter} were previously established in \cite{Zhang95}.

Finally, we end the introduction with a remark on the history.

\begin{rem} As far as we know, there are three ways to show that a smooth Fano manifold does not have an unramified  cover with infinite degree. Chronologically, the first approach is  using Atiyah's $L^2$-index theorem (\cite{Ati76}); the second one is a combination of Yau's solution to the Calabi-Yau conjecture and Myer's Theorem (see \cite{Kob61, Yau78}); and the third one is by the rational connectedness of Fano manifolds (see \cite{Campana92, KMM92}).  The latter two approaches can be used to show that the smooth locus of a log del Pezzo surface has finite fundamental group (see \cite{FKL93} and \cite{KM99}). We do not know how to prove similar statements  along these lines for threefolds. On the other hand, although our approach of using  $L^2$-index Theorem works for terminal Fano threefolds, it is not clear to us how to use it to solve the klt surface case.
\end{rem}

\bigskip
\begin{ack} We would like to thank Jean-Pierre Demailly, Philippe Eyssidieux and Chi Li for helpful discussions. Part of this work was done during the first author's visit to Beijing International Center for Mathematical Research (BICMR). He would like to thank BICMR for the hospitality and stimulating environment. The second author is partially supported by the grant `The (Chinese) National Science Fund for Distinguished Young Scholars'.
\end{ack}
\begin{convention}\label{conv} See \cite{KM98} for the basic definition of the terminologies in birational geometry.

\end{convention}
\section{Preliminary}

\subsection{Deligne-Mumford stack}\label{ss-stack}
In this section, we collect some properties of (algebraic/analytic) Deligne-Mumford (DM) stack. If we do not specify whether it is algebraic or analytic, then the property should hold for both.

Given an analytic DM stack $\mX$, there is an analytic DM stack  $\mZ$ as its {\it universal (unramified) cover}, such that $\mZ \to \mX$ is an unramified cover, and any other unramified cover $\mY\to \mX$ admits a factorization $\mZ\to \mY \to \mX$. We can define $\pi_1(\mX)$ to be the deck transformation group of $\mZ$. If an integral DM stack has trivial isotropic group along the generic point, then we call it an {\it orbi-space}.

 Let $\mathcal{X}$ be a normal separated orbi-space, with a coarse moduli space $\pi:\mathcal{X}\to X$. Near any codimension 1 point $P_i$, it has the form
 $$(X, D_i^\frac{1}{m_i}), \mbox{\ \ where $D_i$ is the image of $P_i$ of on $X$},$$
 and the local isotropic group of $P_i$ is isomorphic to $\mathbb{Z}/m_i \mathbb{Z}$. We call $m_i$ the {\it order of root of $\mathcal{X}$ along $P_i$}.

Now let $\Delta=\sum a_i\overline{P_i}$ an effective $\mathbb{Q}$-divisor on $\mathcal{X}$.   We define
$$D=\sum_{r_i}D_i, \mbox{\ \ where } r_i=\frac{a_i+m_i-1}{m_i},$$
 We define  $(X,D)$ to be the {\it coarse moduli pair} of $(\mX,\Delta)$.

 Conversely, if $(X,D=\sum^k_{i=1} \frac{m_i-1}{m_i}D_i)$ is a simple normal crossing pair. Then there is a smooth orbi-fold  $\mathcal{X}$ which is unique up to 1-isomorphic and realizes $(X,D)$ as the coarse moduli pair of $\mathcal{X}=(\mathcal{X}, 0)$. In this case, there is a surjection on the fundamental groups
$$\pi_1(X\setminus D) \to \pi_1(\mathcal{X}),$$
and the kernels are generated by the elements
$\gamma_1^{m_1}$, ...., $\gamma_k^{m_k}$, where $\gamma_1$, ...., $\gamma_k$  are the loops around $D_1$, $D_2$,..,. $D_k$. We also denote this group by $\pi_1(X,D)$.


 If $(Z^0, E^0)$ and $(X^0, D^0)$ are coarse moduli pairs of orbi-spaces $\mathcal{X}$ and $\mathcal{Z}$. We call $f^0\colon (Z^0,E^0)\to (X^0,D^0)$ an unramified cover  if and only if it is the induced map for a unramified coveri $\mathcal{Z}\to \mathcal{X}$.


\subsection{Stratification theory}\label{ss-ws}
In this section, we give a brief sketch of what we need from the topology of a singular analytic space. The proofs  can be found in \cite{GM88}.

Let $X$ be an analytic space which can be embedded in some smooth manifold $M$. Then there is a {\it Whitney stratification}
$$X_0=X^{\rm sm}\subset X_1\subset X_2\subset \cdots \subset X_m=X, $$
such that
\begin{enumerate}
\item For any $i$, $X_i$ is open in $X$ and $W_i=X_i\setminus X_{i-1}$ is a smooth complex manifold which is closed in $X_i$;
\item There is a tubular neighborhood $U_i$ of $W_i$ in $X_i$,  with a topological deformation retract $\pi\colon U_i\to W_i$, such that for any fiber $w\in W_i$, the fiber $\pi^{-1}(w)$ is isomorphic to the cone over ${\rm Link}(x\in N)$, where $N$ is the intersection of $W$ with a complex submanifold $T$ that is transversal with $W_i$ at $w$ and satisfy $W_i\cap T=\{w\}$.
\item If $X$ is a quasi-projective variety, then $N$ can be chosen to be the intersection of $k={\rm dim}(W_i)$ general hyperplanes with $U_i$.
\item If $(X,D)$ is a pair, we can choose a Whitney stratification which is compatible with both $X$ and $D$.
\end{enumerate}

\section{Local fundamental group}

\subsection{Retraction to a divisor}
Consider a proper effective Cartier divisor $E$ in $X$. Fix a smooth metric, we denote by $U_{\delta}$ a closed tubular neighborhood of $E$ consisting of all points with distance less than or equal to $\epsilon$ from $E$. In this section, we study the map $\pi: U_{\delta}\to E$  constructed in \cite{Gor81} and further developed in \cite{GM83, GM88}.

Fix a Whitney stratification, we assume that the restriction of $\mathcal{O}(E)$ on each strata is isomorphic to the trivial bundle.  By \cite[II.5.A]{GM88}, we know that the homotopy class of $U_{\delta}$ does not depend on the choice of metric for $\epsilon$ sufficiently small. Furthermore, as constructed in \cite[Section 7]{Gor81}, there is a morphism
$$h\colon U_{\delta}\times [0,1] \to U_{\delta},$$
such that $h(\cdot, 0)={\rm id},$
$$\pi:=h(\cdot,1)  \colon U_{\delta} \to E,$$ and $h|_{E\times [0,1]}$ gives an isotopy of $E\to E$ between ${\rm id}$ and $ h(\cdot, 1)|_{E} $.

Fix a stratum $A\subset E$ which is of codimension $c$ and consider the $\epsilon_0$-interior $A^{\circ}$ defined in \cite[Page 180]{GM83} for sufficiently small $\epsilon_0$.  Denote by $V_{A^{\circ}}=\pi^{-1}(A^{\circ})$  the closed tubular neighborhood of  $A^{\circ}$, $V^0_{A^{\circ}}=V_{A^{\circ}}\setminus A^{\circ}$ the punctured tubular neighborhood  and let $\partial V^0_{A^{\circ}}$ be the boundary of $V^0_{A^{\circ}}$ in  $U^0:=U_{\delta}\setminus E$. Since we assume that $\mathcal{O}(E)$ is trivial on $A$, there is a morphism
$$ \rho_A: \pi^{-1}(A)\to \mathbb{D},\mbox{ \  where }\mathbb{D}=\{t\ | \ |t|\le \epsilon\} $$
for sufficiently small $\epsilon$, such that
$$\rho_A^{-1}(0)=E \cap \pi^{-1}(A).$$
Furthermore, we assume that $\rho$ yields a topological fibration over $\mathbb{D}^0=\mathbb{D}\setminus \{0\}$.

Then we apply the argument in  \cite[Section 6]{GM83} (also see \cite[Part 2, 6.13]{GM88}) to conclude that after possibly choosing a smaller $\epsilon$, for any $p\in A^{\circ}$,
$$\left( L:= \pi^{-1}(p)\setminus \{p\}, \partial L \right)$$ is homotopic to a bundle over $\mathbb{D}^{\circ}$, where $\mathbb{D}^{\circ}=\{t\ | \ 0< |t|\le\epsilon\}$, with a fiber homotopic to a collared affine variety which is of complex dimension $c$.


 Let $Z\subset E$ a closed subset which is a union of strata with complex codimension $c$. Denote by $W$ the complement of $Z$ in $E$. Let $V:=\pi^{-1}(W)\subset U_{\delta}$ and $V^0=V\setminus E$.  Then $V^0$ is an open analytic subspace of $ U^0=U_{\delta}\setminus E$ and its closure is obtained by adding a collared boundary.

Using the stratified theory analogue to \cite[II.6.13]{GM88},  we can show the following technical lemma, which is our main tool to show Theorem \ref{thm-lf}.
\begin{lem}\label{l-codim2}
Assume $c\ge 2$, then
$$\pi_1(V^0)\to \pi_1( U^0) $$
is surjective.
\end{lem}
\begin{proof} 
We use the notations as above. So we can write
$$ E\setminus Z: =E_0\subset E_1\subset \cdots \subset E_m=E,$$
such that both $E_i$ and $E_{i+1}$ are unions of strata and  are open in $E$. Furthermore, for any fixed $i$, we assume $E_{i+1}=E_i\cup A$ where $A$ is a stratum. In particular, ${\rm Codim}_A(E)\ge c\ge 2$.
We can write
$$ \overline{V}_{A^{\circ}} \supset \partial V_{A^{\circ}} =\mathcal{L}_1\cup \mathcal{L}_2 \cup \mathcal{L}_3,$$
where $\mathcal{L}_1=\pi^{-1}(\partial A^{\circ})$, $\mathcal{L}_2$ is the horizontal part, i.e.,  intersection $\overline{V}_{A^0}\cap {\pi^{-1}(E_i)}$, and  $\mathcal{L}_3$ is the vertical part, i.e., $\partial V_{A^{\circ}}\cap \rho_A^{-1}(\{t\ |\ |t|=\epsilon\})$.

We note that $\mathcal{L}_1$ and $\mathcal{L}_3$ are boundaries of a collared space, and by Van Kampen's theorem, to show that
 $$\pi_1(\pi^{-1}(E_i)\setminus E_i)\to\pi_1( \pi^{-1}(E_{i+1})\setminus E_{i+1})$$
  is surjective, it suffices to show that
  $$\pi_1(\mathcal{L}_2)\to \pi_1(  V^0_{A^{\circ}})$$
is surjective.

Let $L_2$ be the fiber of $\mathcal{L}_2\to A^{\circ}$. There is a commutative diagram
\[
\begin{CD}
\pi_1( L_2)@>>> \pi_1(\mathcal{L}_2) @>>> \pi_1( A^{\circ}) @>>> \{e\}\\
@VVV                              @VVV                       @|                  @.\\
\pi_1(L)@>>> \pi_1(V^0_{A^{\circ}}) @>>> \pi_1(A^{\circ}) @>>> \{e\},
\end{CD}
\]
where the rows are exact sequences.
Thus it suffices to show that
$$\pi_{1} ( L_2)\to \pi_1(L)$$
is surjective.

Denote by $X_{t,A^{\circ}}$ to be  $V^0_{A^{\circ}}\cap \rho_A^{-1}(t)$ for $0<|t|\le \epsilon$. We note that $L$ (resp. $L_2$) is an  $X_{t,A^{\circ}}$-bundle (resp. $\partial X_{t,A^{\circ}}$-bundle) bundle over $\mathbb{D}^{\circ}$. Thus it suffices to show that
$$\pi_{0} (\partial X_{t,A^{\circ}})\cong \pi_0(X_{t,A^{\circ}})\mbox{\ \  and \ \ }\pi_{1} (\partial X_{t,A^{\circ}}) \twoheadrightarrow \pi_1(X_{t,A^{\circ}}).$$

By the previous discussion, we know that $X_{t,A^{\circ}}$ is homotopic to a collared affine analytic space with dimension at least $c$ (cf. e.g. \cite[Part II, 6.13.5]{GM88}).
In particular, $$\pi_{i} (\partial X_{t,A^{\circ}})\to \pi_i(X_{t,A^{\circ}})$$
is an isomorphism for $i<c-1$ and is surjective for $i=c-1$ (cf. \cite[Part II, 6.13.6]{GM88}). Since we assume $c>1$, we exactly have
$$\pi_{0} (\partial X_{t,A^{\circ}})\cong \pi_0(X_{t,A^{\circ}})\mbox{\ \  and \ \ }\pi_{1} (\partial X_{t,A^{\circ}}) \twoheadrightarrow \pi_1(X_{t,A^{\circ}}).$$
\end{proof}

\subsection{Koll\'ar component}
Let $ x\in X$ be a klt singularity.
It follows from \cite[Lemma 1]{Xu14} that there is a morphism $Y\to X$  extracting the Koll\'ar component $E$ as the only exceptional divisor which admits a log Fano pair structure. More precisely, there exists $f:Y\to X$, such that ${\rm Ex}(f)=E$ which is an irreducible divisor, $f(E)=x$, $(Y, E)$ is plt (with $E$ being the only log canonical center) and
$$-(K_Y+E)\sim_{\mathbb{Q},X}-aE$$ is ample over $X$ for some  rational $ a>0$. We can choose $ U_{\delta} $ and $U^0$ as in the previous section. In particular,
$$\pi_1({\rm Link}(x\in X))\cong \pi_1( U^0). $$

Write $(K_Y+E)|_E=K_{E}+\Delta_E$, where $\Delta_E=\sum\frac{r_i-1}{r_i}\Delta_i$ is the different divisor (e.g. \cite[Definition 4.2]{Kollar13}). Since $(Y, E)$ is plt, by adjunction, we know that $(E,\Delta_E)$ is klt.

Fix a Whitney stratification. Let $Z$ be the union of strata on $E$ which is of codimension at least $2$ and let  $E^0\subset E$ be the complement of $Z$. We can assume $\Delta_E|_{E^0}$ is smooth. Then ${\rm Codim}_{E}(E\setminus E^0)\ge 2$, and $(E^0, \Delta_{E^0}:=\Delta_E|_{E^0})$ is the coarse moduli pair of an orbifold $\mathcal{E}^0$ (cf. Section \ref{ss-stack}).

\begin{lem}Denote by $V^0$ to be $\pi^{-1}(E^0)\setminus E^0$ as in the previous section. After possibly shrinking $E^0$ to a big open set,  we have a differentiable fiber bundle morphism
$$0\to \mathbb{D}^{\circ}\to V^0\to \mathcal{E}^0 \to 0.$$
\end{lem}
\begin{proof} By the formula of different divisor to compute $\Delta$, we know that along the general point $p$ of a divisor $S$ on $E$, a chart $U_p$ for an analytic local neighborhood of $p\in Y$ can be
given by $$(p\in (Y,E)) \cong \left(0\in \left(\mathbb{C}^2=(x,y), (x=0)\right)/(\mathbb{Z}/m\mathbb{Z}) \times \mathbb{C}^{n-2}\right ),$$
 and the coefficient of $S$ in $\Delta_E$ is precisely $\frac{m-1}{m}$.
 Using this coordinates to calculate, it is clear that we have
$$0\to \mathbb{D}^{\circ}\to U_p\to  \left(\mathbb{C}, \{0\}^{\frac{1}{m}}\right) \times \mathbb{C}^{n-2} \to 0,$$
where $U_p= \left (\{(x,y)\in \mathbb{C}^2|\ 0<|x| \le \epsilon \}/(\mathbb{Z}/m\mathbb{Z}) \right)\times \mathbb{C}^{n-2}$ which can be use to represent $V^0$ around $p$ and $(\mathbb{C}, \{0\}^{\frac{1}{m}})$ is the root stack with $m$-th root along $\{0\}$.
\end{proof}
\begin{proof}[Proof of Theorem \ref{thm-lf}]If we assume Conjecture \ref{conj-gf} holds in dimension $n-1$, then we know that $\pi_1(\mathcal{E}^0)$ is finite. The same argument as in the last paragraph of the proof of \cite[Theorem 1]{Xu14} implies that $\pi_1(V^0)$ is finite. Since we have $\pi_1(V^0)$ surjects to $\pi_1(U^0)$ by applying Lemma \ref{l-codim2} to a multiple of $E$ which is Cartier, we conclude that  $\pi_1(U^0)$ is finite.
\end{proof}
\begin{proof}[Proof of Theorem \ref{cor-3l}]Since Conjecture \ref{conj-gf} holds for surfaces (cf. \cite{FKL93}), Theorem \ref{cor-3l} is implied by Theorem \ref{thm-lf}.
\end{proof}

Using the standard Whitney stratification theory, we can show that the finiteness of the smooth locus of the local fundamental group indeed implies the finiteness of the fundamental group of the smooth locus of a germ.
\begin{lem}\label{l-sl}
Assuming Conjecture \ref{conj-lf} holds for dimensional less or equal to $n$. Let $x\in X$ be an algebraic klt singularity. Then if we choose a small Euclidean neighborhood $U$ of $x$ in $X$, then $\pi_1(U^{\rm sm})$ is finite.
\end{lem}
\begin{proof}Consider a Whitney stratification of a neighborhood of $x$ given as in Section \ref{ss-ws}.
Let
$$X_0=X^{\rm sm}\subset X_1\subset X_2\subset \cdots \subset X_m=X, $$
such that $X_i$ is an open set of $X$ and $W_i=X_i\setminus X_{i-1}$ is a strata. We will show that the natural morphism
$$\pi_1(X_{i-1}\cap U)\to \pi_1(X_i\cap U)$$
has finite kernel.  Let $U_i$ be a tubular neighborhood of $W_i$ as in Section \ref{ss-ws}.
Applying Van Kampen theorem, we know that it is enough to show that
$$\rho\colon \pi_1(U^0_i) \to \pi_1(U_i)=\pi_1(W_i)$$
has finite kernel.

However, there is a surjection from
$\pi_1(N^0_w=N_w\setminus \{w\})$ to the kernel ${\rm Ker}(\rho)$ for any $w\in W_i$, where $N_w$ is given in (2) in Section \ref{ss-ws}. Because $N$ can be chosen to be a general intersection of $k=\dim(W_i)$ hyperplanes, $(w\in N)$ is an algebraic klt singularity. Then it follows from our assumption that Theorem \ref{thm-lf} holds for dimension less or equal to $n$. Therefore, we know that  $\pi_1(N^0_w)$ is finite.
\end{proof}

In the latter application, we need the following variant.
\begin{lem}\label{l-sm}
Assume Conjecture \ref{conj-gf} holds for dimension at most $n-1$. Let $(x\in (X,D))$ be an $n$-dimensional klt singularity in a pair with $D=\sum \frac{n_i-1}{n_i}D_i$.  Assume each $D_i$ is $\mathbb{Q}$-factorial.  Denote by $X^0$ to be a large open set of $X$, such that $(X^0, {\rm Supp}( D)|_{X^0})$ is simple normal crossing. Then $\pi_1(X^0, D^0=D|_{X^0})$ is finite.
\end{lem}
\begin{proof}
Since $x\in (X,\Delta)$ is local, we can choose a ramified cover such that there is an algebraic singularity $y\in Y$, with a finite morphism $f\colon Y\to X$ such that $f^*(K_X+\Delta)=K_Y$ (see \cite[Lemma 6.3]{MP04}).
Then it is clear $\pi_1(Y^{\rm sm})\to \pi_1(X^0, \Delta^0)$ has a finite cokernel. So this follows from Theorem \ref{thm-lf} and Lemma \ref{l-sl}.
\end{proof}

\subsection{Construction of covering}\label{s-covering}

In this section, we aim to prove the following statement.

\begin{prop}\label{thm-cover}
Let $(X,D)$ be a $n$-dimensional klt quasi-projective pair, such that $D=\sum_{i}\frac{n_i-1}{n_i}D_i$.Assume each $D_i$ is $\mathbb{Q}$-Cartier.  Let $X^0\subset X$ be a smooth open set such that ${\rm Codim}_X(X\setminus {X}^0)\ge 2$ and $D^0:=D|_{X^0}$ are smooth. Consider the pair $(X^0, D^0)$ as an orbifold and let $\phi^0\colon (Z^0,E^0) \to (X^0,D^0)$ be an unramified Galois cover of analytic spaces with Galois group $G$ (as a quotient of the orbifold fundamental group).

Assume that Conjecture  \ref{conj-gf} holds for dimension at most $n-1$. Then there is a normal analytic space $Z$ with a properly discontinuous action by $G$, such that
 ${\rm Codim}_{Z}(Z\setminus Z^0)\ge 2$ and and if we  denote by $E$ the extension of $E^0$,
the quotient pair of $(Z,E)$ by $G$ is  $(X,D)$.
\end{prop}

\begin{proof}
We denote the branch order of $Z^0\to X^0$ by $m_i$ along the component $D^0_i$ of $D^0$. We can then assume that $Z^0$ is given by the underlying space of the universal cover for the orbifold $(X^0, \tilde{D}^0= \sum (D^0_i)^{\frac{1}{m_i}})$ and write
$$(\phi^0)^*(K_{X^0}+D^0)=K_{Z^0}+E^0.$$
In fact, if we can extend the universal cover $(Z^0,E^0)$ to $(Z,E)$, then we can extend the intermediate covers by simply taking the quotient of $(Z,E)$.

Now fix a Whitney stratification of $(X,D)$, and assume
$$X_0=X^0\subset X_1\subset X_2\subset \cdots \subset X_m=X, $$
such that $X_i$ is an open set of $X$ and $W_i=X_i\setminus X_{i-1}$ is a strata. We assume that we have constructed $Z_{i-1}$ with a $G$-action admitted a morphism $\phi_{i-1}: Z_{i-1}\to X_{i-1} $, such that
$(Z_{i-1}, E_{i-1})/G$ has $(X_{i-1}, D|_{X_{i-1}})$ as its coarse moduli space.

We  take a covering of $W_i$ by simply connected Euclidean open neighborhoods $\{B_i\}$. We know that each component $V^0_{i-1,k}$ of $\phi^{-1}_{i-1}(U^0_{i-1,k})$ for  $U^0_{i-1,k}:=\pi^{-1}_i(B_k)\setminus B_k\subset X_{i-1}$  corresponds to the morphism
$$\pi_1(U^0_{i-1,k},\Gamma_{i-1,k}^0) \twoheadrightarrow G^* \hookrightarrow \pi_1(X^0,\tilde{D}^0)=G,$$
where $\Gamma_{i-1,k}^0=E|_{U^0_{i-1,k}}$.

By our assumption that Conjecture \ref{conj-lf} holds for dimension at most $n$, we claim that $\pi_1(U^0_{i-1,k},\Gamma_{i-1,k}^0)$ is finite. In fact, since $B_k$ is simply connected, there is surjection
$$\pi_1(N^0,\Theta^0)\to \pi_1(U^0_{i-1,k},\Gamma_{i-1,k}^0) ,$$
where $(N,\Theta)$ is the neighborhood of $x$ in the intersection of $k=\dim B_k$ hyperplanes with $(T^{-1}_i(B_k), D|_{T^{-1}_i(B_k)})$ as in Section \ref{ss-ws} and
$$(N^0,\Theta^0)=(N,\Theta)\cap Z^0.$$ Thus $x\in (N,\Theta)$ is again a klt algebraic singularity of dimension $(n-\dim W_i)$. So the claim follows from the conjectural assumption and Lemma \ref{l-sm}.

Therefore, $G^*$ is finite, which in turn implies that the morphism $V^0_{i-1,k}\to U^0_{i-1,k}$ is finite. Thus, there is a unique normal analytic space $V_{i-1,k}$ containing $V^0_{i-1,k}$ as a Zariski open set, such that $V_{i-1,k}\to U_{i-1,k}$ is a finite morphism and extends $V^0_{i-1,k}\to U^0_{i-1,k}$.

For a fixed $k$, we can glue $ Z_{i-1}$ with $V_{i-1,k}$ along $V^0_{i-1,k}$ for all components $V^0_{i-1,k}$. The resulting complex analytic space clearly admits a $G$ action. And then we glue for all $k$ to get $Z_{i}$.

By our construction, the analytic orbi-space $[Z_i/G]$ has a proper birational morphism to $ X_i$. In fact, we just glue together the open set of $[Z_{i-1}/G]$ which is proper over $X_{i-1}$ and $[V_{i-1,k}/G^*]$ which is proper over $U_{i-1,k}$.

If we denote by $(\mathcal{X}_i,\Delta_i)=([Z_i/G],[E_i/G])$, then $(X_i,D_i)$ is the coarse moduli pair of $(\mathcal{X}_i,\Delta_i)$, since $X_i$ is the coarse moduli space of $\mathcal{X}_i$. It is clear from our construction that ${\rm Codim}_Z(Z\setminus Z^0)\ge 2$, thus the equality on the branched divisor only needs to be verified  on $Z^0\to X^0 $, which is true by definition of $Z^0$.
\end{proof}
\section{Finiteness of $\pi_1(X^{\rm sm})$ for canonical Fano threefold $X$}

\subsection{$L^2$-index theorem and vanishing results}
In this section we review the $L^2$-index theorem (\cite{Ati76}, \cite{CD01}, \cite{Eyssidieux}) and some vanishing results of the $L^2$-cohomology (\cite{Demailly82}, \cite{Demailly92}).

 Given a complex analytic space $\tilde{X}$ together with a properly discontinuous action of a countable discrete group $\Gamma$ such that the quotient is compact, let $\tilde{\mathcal{F}}$ be an analytic coherent sheaf  on $\tilde{\mathcal{X}}$ admitting an action of the group $\Gamma$.
 One can define the $L^2$-cohomology group $H^i_{(2)}(\tilde{X}, \tilde{\mathcal{F}})$.
 If $\tilde{\mathcal{F}}$ is locally free, then the group can be computed using the $L^2$-Dolbeault complex of differential forms with values in $\tilde{\mathcal{F}}$.
The $L^2$-cohomology groups satisfy some usual properties of quasi-coherent sheaves. For example, a short exact sequence of sheaves induces a long exact sequence of $L^2$-cohomology groups. If there is a $\Gamma$-equivariant proper morphism between two spaces $p: X \to Y$, there is a spectral sequence with $E_2$ terms $H^i_{(2)}(Y, Rp^j_*(\tilde{F}))$ abutting to $H^{i+j}_{(2)}(X, \tilde{F})$.

 These $L^2$-cohomology groups are usually infinite dimensional.
 However, as is shown in \cite{Ati76}, \cite{CD01}, \cite{Eyssidieux}, one can assoicate a real number $h^i_\Gamma(X, \tilde{\mathcal{F}})$ as a ``renormalized" dimension.
 We define the $L^2$-index of $\tilde{\mathcal{F}}$ as
 $$
 \chi_{(2)}(\tilde{X}, \tilde{\mathcal{F}})=\sum_i (-1)^i h_{\Gamma}^i(\tilde{X}, \tilde{\mathcal{F}}).
 $$

 Denote by  $X$ the quotient of the action. When the group action is free, specifying the sheaf $\tilde{\mathcal{F}}$ and the action of $\Gamma$ on $\tilde{F}$ is equivalent to giving a coherent sheaf $\mathcal{F}$ on $X$.
 If furthermore $X$ is projective or compact K\"ahler, then $$\chi_{(2)}(\tilde{X}, \tilde{\mathcal{F}})=\chi(X, \mathcal{F}).$$
 In particular if the group $\Gamma$ is a finite group, then  $$\chi_{(2)}(\tilde{X}, \tilde{\mathcal{F}})=\chi(\tilde{X}, \tilde{\mathcal{F}})/\vert \Gamma \vert.$$
 If the group action is not free, one can take the analytic orbispace quotient $\mathcal{X}$ and the coherent sheaf $\mathcal{F}$ on $\mathcal{X}$.
 In general the $L^2$-index is not the same as the Euler characteristic of the sheaf $\mathcal{F}$ and there are contributions coming from fixed points/orbifold points.
 In the following we give an explicit $L^2$-index formula in the presence of fixed points in a special case.

 \begin{thm}\label{orbifoldRR}
 Let $Z$ be a three dimensional complex analytic space and $\Gamma$ a countable discrete group acting on $Z$.
 Assume that the action is properly discontinuous and has only isolated fixed points.
 Denote by $\mathcal{X}$ the DM stack $[Z/\Gamma]$ and $X$ its coarse moduli space.
 Assume that $X$ is projective with terminal singularities. Then we have the following formula to calculate the $L^2$-index on $Z$:
 $$\chi_{(2)}(Z,\mathcal{O}_Z)=-c_2(X)\cdot K_X+\sum\frac{1}{24|I_i|}\sum_j (r_{ij}-\frac1{r_{ij}})\geq -c_2(X)\cdot K_X,$$
 where the sum is taken over all the singularities $P_i$ on $\mathcal{X}$, $I_i$ is the inertial group, $r_{ij}$ is the index of a singularity in the basket, and the second sum is took over all singularities in the associated basket of $P_i$.
 \end{thm}

In the statement of the theorem, by $c_2(X) \cdot K_X$ we mean taking the closure of a cycle representing the Chern class of the tangent sheaf of the smooth locus of $X$ and intersecting it with $-K_X$. Since $-K_X$ is $\mathbb{Q}$-Cartier and the singular locus consists of isolated points, this intersection number is well-defined. For the precise definition of the associated basket of a $3$-fold canonical singularity, we refer the reader to \cite[Section 10]{Reid87}.

For the proof of the theorem, we first note the following lemma, which is well-known to experts but we do not know a precise reference. We are grateful to P. Eyssidieux for providing the argument. For a similar formula, see \cite[(6.2)]{WW13}.

\begin{lem}[Eyssidieux]\label{L2-index}
Let $Z$ be a smooth complex analytic space together with a properly discontinuous action of a countable discrete group $\Gamma$. Assume that the quotient $Z/\Gamma$ is compact. Let $D$ be a fundamental domain of the action such that the non-free locus of the action are either in the interior of $D$ or outside the closure of $D$. Then we have
\[
\chi_{(2)}(Z, \mathcal{O}_Z)=\int_D Td(Z)=\int_{[Z/\Gamma]} Td([Z/\Gamma]),
\]
where $Td(X)$ is the analytic Todd class (i.e. a differential form) constructed using a $\Gamma$-invariant metric on $Z$, and $[Z/\Gamma]$ denotes the orbifold quotient.
\end{lem}

\begin{proof}[Sketch of Proof]
The first equality basically follows from Getzler's proof of the index theorem using the heat kernel.
The group $\Gamma$ acts on the space of $L^2$-sections of the heat kernel, which is a $\Gamma$-equivariant vector bundles.
Then the Von Neumann trace is the integral of the trace of the kernel restricted to the diagonal.
So we can express the $L^2$-index as an integral of the Todd class over the fundamental domain $D$.

The second equality is obvious as the two integrals  only differ by a zero measure set.
\end{proof}

When the group $\Gamma$ is finite, we have the following special case.

\begin{cor}\label{finite-index}
Let $Z$ be a compact complex analytic space with isolated rational singularities. Assume that there is a finite discrete group $\Gamma$ acting on $Z$ and the action is free away from isolated points. Then
\[
\chi_{(2)}(Z, \mathcal{O}_Z)=\chi(Z, \mathcal{O}_Z)/|\Gamma|.
\]
\end{cor}
\begin{proof}
Choose a fundamental domain $D$ of the action such that all the isolated non-free points and singlar points are contained in the interior of $D$. Let $p: Z' \to Z$ be a $\Gamma$-equivariant resolution of singularities. Then $D'=p^{-1}(D)$ is a fundamental domain of the action on $\Gamma$. By Lemma \ref{L2-index}, we have
\begin{align*}
&\chi_{(2)}(Z', \mathcal{O}_{Z'})\\
=&\int_{D'} Td(Z')\\
=&\frac{1}{|\Gamma|}\sum_{\gamma \in \Gamma} \int_{\gamma{D'}} Td(Z')\\
=&\frac{1}{|\Gamma|}\int_{Z'} Td(Z')\\
=&\frac{1}{|\Gamma|}\chi(Z', \mathcal{O}_{Z'}).
\end{align*}

Since $Z$ has rational singularities,
\[
\chi_{(2)}(Z', \mathcal{O}_{Z'})=\chi_{(2)}(Z, \mathcal{O}_{Z})
\]
and
\[
\chi(Z', \mathcal{O}_{Z'})=\chi(Z, \mathcal{O}_{Z}).
\]
Hence the statement is proved.
\end{proof}

Now we prove Theorem \ref{orbifoldRR}. For more background of our calculation in the proof, see \cite[Section 10]{Reid87}.

 \begin{proof}[Proof of Theorem \ref{orbifoldRR}]
Consider a resolution of singularities (as a stack) $\mathcal{X}'\to \mathcal{X}$.
Then by base change, we know that we obtain a resolution $Z'\to Z$.
By the $L^2$-index theorem,  we know that
 $$\chi_{(2)}(Z,\mathcal{O}_Z)=\chi_{(2)}(Z',\mathcal{O}_{Z'})=\int ch(\mathcal{O}_{\mathcal{X}'})\cdot Td(T_{\mathcal{X}'}).$$
 The first equality follows from the fact that $Z$ only has rational singularities. The second equality is Atiyah's $L^2$-index formula (for orbifolds) (Lemma \ref{L2-index}).


 So it suffices to compute the difference between
$  ch(\mathcal{O}_{\mathcal{X}'})\cdot Td(T_{\mathcal{X}'})$ and $-K_X\cdot c_2(X)$.
This is a sum of contributions from each singularity, and each contribution only depends on the analytic types of the singularities on $\mathcal{X}$.
In the following we compute the contribution of each singularity individually. 

Given a singular point $x \in X$, by the classification of terminal singularities (see, e.g. \cite[6.1]{Reid87}),
we know that locally it is of the form $U/G$ for some finite abelian group $G$, where $U \subset \mathbb{C}^4$ is a germ of a threefold terminal singularity of index $1$.
Furthermore, let $z$ be a point in $Z$ lying over the point $x$.
Then locally around the point $z$, the analytic space $Z$ is of the form $U/H$, where $H$ is a finite abelian subgroup of $G$.
The isotropy group $\Gamma_z$ at the point $z$ is the quotient group $G/H$.

For each terminal singularity of the form $U/G$, there is a family of $\mathbb{Q}$-smoothing $$\mathcal{U}/G \to T,$$
where $\mathcal{U} \to T$ is a $G$-equivariant smoothing of the singularity.
We can find a projective family $$\bar{\mathcal{U}} \to T$$ such that the group $G$ acts on the total space $\bar{\mathcal{U}}$ and induces an action on each member.
To see this first note that $U$ is a hypersurface in $\mathbb{C}^4$ and the $G$-action extends to $\mathbb{C}^4$.
By \cite[8.4]{Reid87} there is a smooth $4$-fold $W$ with a $G$-action such that the fixed points of the action are either isolated or one dimensional and the action is free away from the fixed point loci. Furthermore the local action of $G$ around each fixed point is the same as the action on $\mathbb{C}^4$.
We can take a general $G$-invariant hypersurface of sufficiently high degree passing through finitely many fixed points, such that
at one of the fixed points, the hypersurface is singular with the same type of singularity as $U$ and at all the other fixed points the hypersurface is smooth.
When the singular locus consists of isolated points, we may choose the hypersurface to contain only one fixed point.
A general $G$-equivariant smoothing of the hypersurface containing the singular fixed point gives the family $\bar{\mathcal{U}} \to T$.
We remind the readers that by construction the action on every member of the family of hypersurfaces is free away from the finitely many isolated fixed points.

The family $\bar{\mathcal{U}} \to T$ gives rise to a family $\bar{\mathcal{U}}/H \to T$ together with an action of $G/H$. The latter family $\bar{\mathcal{U}}/H \to T$ has the property that a general member has cyclic quotient singularities and there is one member which has one singularity of the type $U/H$. The quotient by $G/H$ of the family $\bar{\mathcal{U}}/H \to T$ is isomorphic to the family $\bar{\mathcal{U}}/G \to T$. Every singularity of a general member of the family has the same type of quotient singularities as $\mathbb{C}^3/G$. There is one member of the family which has two types of singularities. One of the singularities is of the form $U/G$ and there is only one such singular point. All the other singularities are quotient singularities of the same type $\mathbb{C}^3/G$.

Since both of the terms
$$\chi_{(2)}(\bar{\mathcal{U}}_t/H, \mathcal{O}_{\bar{\mathcal{U}}_t/H})=\chi(\bar{\mathcal{U}_t}/H, \mathcal{O}_{\bar{\mathcal{U}_t}/H})/|G/H|$$ and $-K_{\bar{\mathcal{U}_t}/G} \cdot c_2(\bar{\mathcal{U}_t}/G)$ are deformation invariant,
we may compute the contribution of each singularity by summing over the contributions of all the singularities of a general member (i.e. the basket of singularities)
and hence assume that the singularities are cyclic quotient singularities.

Let $V$ be a general member of the family $\bar{\mathcal{U}} \to T$ and let $m$ be the number of fixed points of $V$ under the subgroup $H$.
Recall that a point in $V$ is either fixed by the whole group $G$ or has an orbit of $|G|$ points.
So a point is fixed by $H$ if and only if it is fixed by $G$.
We have the quotient projective varieties $V/H$ and $V/G$. The variety $V/H$ admits a $G/H$ action and the quotient morphism $V/H \to V/G$ locally models the quotient morphism $Z \to X$ around a singular point in $X$.
We have
\begin{align*}
\chi_{(2)}(V/H, \mathcal{O}_{V/H})&=\frac{\chi(V/H, \mathcal{O}_{V/H})}{|G/H|} (\text{Corollary }\ref{finite-index})\\
&=\frac{-c_2(V/H)\cdot K_{V/H}}{|G/H|}+m \cdot \frac{|H|^2-1}{24 |H| \cdot |G/H|}\\
&={-c_2(V/G)\cdot K_{V/G}}+m \cdot \frac{|H|^2-1}{24 |H| \cdot |G/H|}.
\end{align*}
Here the second equality follows from \cite[10.2, 10.3]{Reid87}. Thus the contribution is of the form stated.
\end{proof}
\bigskip

The following theorem is essentially proved in \cite{Demailly82} and \cite{Demailly92} using $L^2$-estimates.

\begin{thm}[{\cite[5.1]{Demailly82}, \cite{Demailly92}, see also \cite[4.5]{CD01}}]\label{L2vanishing}
  Let $Z$ be a complex analytic space together with a properly discontinuous action of a countable discrete group $\Gamma$.
  Denote by $\mathcal{X}$ the quotient stack $[Z/\Gamma]$ and $(X, \Delta)$ its coarse moduli pair. Assume that $X$ is projective and $(X, \Delta)$ has klt singularities.
  Let $L$ be a line bundle on $X$ such that $L\equiv K_X+\Delta+N$, where $N$ is big and nef. Denote by $\tilde{L}$ the pull-back of $L$ to $Z$.
  Then
  $$
  H^i_{(2)}(Z, \tilde{L})=0,  \text{ for all } i>0.
  $$
 \end{thm}

 Two things are slightly different, namely \cite{Demailly82} and \cite{Demailly92} work with smooth K\"ahler manifolds and their unramified covers.
One could use the spectral sequence and a standard argument to reduce the problem to the case of smooth K\"ahler manifolds.
Similarly, the same result can also be generalized for a properly discontinuous action,
provided that one choose the K\"ahler metric for $\tilde{X}$ and the singular Hermitian metric of $\tilde{L} $ as the pull-back metric from the orbispace $\mathcal{X}$ and the line bundle $L$. In fact the key ingredient of the proof of the vanishing is an $L^2$-estimate (\cite[5.1]{Demailly82}), which does not involve the $\Gamma$-action at all. The rest of the argument works in the same way.

 \subsection{Fano threefolds}
  In this section, we aim to prove Theorem \ref{thm-3ter}. We first collect some basic properties of three dimensional singularities (see \cite[Section 5.3]{KM98}).

  Let $(x, X)$ be a threefold canonical singularity.  By Corollary \ref{cor-3l}, we know that any canonical three dimensional singularity has a finite local fundamental group, and if we let $X'\to X$ be a finite morphism which is \'etale in codimension 1, then $X'$ also only has canonical singularities (see \cite[5.21]{KM98}).

 Let $Z^0\to X^{\rm sm}$ be the universal covering, then Proposition \ref{thm-cover} becomes an unconditional result and we know that there is a properly discontinuous action $G=\pi_1(X^{\rm sm})$ on a three dimensional normal analytic space $Z$, such that $Z/G=X$, and $Z\to X$ is only possibly branched over ${\rm Sing(X)}$. Let $\mathcal{X}$ be the algebraic stack whose analytification gives $[Z/G]$, then by what we have just discussed, $\mathcal{X}$ only has canonical singularities.

 We first treat the case that $X$ only has terminal singularities.

 We have the following result of Miyaoka (see e.g. \cite[3.10]{MP97}).
 \begin{lem}\label{l-c2}
 Let $S$ be a smooth projective surface and $E$ a generically semi-positive vector bundle of rank $r\ge 2$ with $D=c_1(E)$ big and nef.
 Then $c_2(E)$ has positive degree.
  \end{lem}
 \begin{proof}
Denote the Harder-Narashimhan filtration by $F^iE$ and denote by $G^i=F^iE/F^{i-1}E$.
 Then
 $$c_2(E)=\sum_i c_2(G_i)+\sum_{i<j}c_1(F^iE)\cdot c_1(F^jE) .$$
 Denote by $r_i$ the rank of the sheaf $G_i$. Using Bogomolov's inequality for stable sheaves, and the fact that $G_j\subset G_j^{**}$, which is locally free on $S$,
 we know that the right hand side is at least
 $$\sum_i(\frac{r_i-1}{2r_i}-\frac{1}{2})c_1(G_i)^2+\frac{1}{2}c_1(E)^2.$$
 Let  $c_1(G_i)=r_i(a_iD+\Delta_i)$, where $\Delta_i$ is orthogonal to $D$. Computing the intersection number with $D$ shows that $\sum a_i r_i=1$.
Furthermore $a_i\geq 0$ since $E$ is generically semi-positive.
 Hence we know that the above formula is equal to
 $$-\sum_{i}\frac{r_i}{2}(a_i^2D^2+\Delta_i^2)+\frac12D^2$$ which is at least $$\frac{1}{2}(1-\sum_{i}a_i^2r_i)D^2$$ as $\Delta_i^2\le 0$ by Hodge index theorem.
 Since $\sum a_i r_i=1$ and $\sum r_i=r \ge 2$,
we know that $\sum a_i^2r_i<1$. Furthermore $D$ is big and nef, so $D^2>0$. Therefore,  $c_2(E)$ has positive degree.
 \end{proof}
\begin{cor}\label{weakfano}
Let $X$ be a weak Fano threefold with terminal singularities.
 Then $-c_2(X) \cdot K_X>0.$
\end{cor}
\begin{proof}
By Lemma \ref{l-c2}, it suffices to prove that for a very ample divisor $H$, $T_X|_H$ is generically semi-positive.

But this follows from that  the proof of \cite[Theorem 1.2(1)]{KMMT00} which says that for any quotient $E$ of the tangent sheaf $T_X$, $E\cdot H_1\cdot H_2\ge 0$ for very ample divisors $H_1$ and $H_2$.
\end{proof}
To sum up, we have proved the following.
\begin{cor}\label{cor:chi2}
 Let $X$ be a weak Fano threefold with terminal singularities and $Z$ be the space constructed in Section \ref{s-covering}. Then $\chi_{(2)}(Z, \mathcal{O}_Z)>0$.
\end{cor}

Now we can finish the proof. By the $L^2$-vanishing Theorem \ref{L2vanishing}, the $L^2$-cohomology groups $H^i_{(2)}(Z, \mathcal{O}_Z)=0$ for all positive $i$.
Then by Corollary \ref{cor:chi2}, we know that $H^0_{(2)}(Z, \mathcal{O}_Z)$ is nontrivial.
Thus there is an $L^2$-integrable holomorphic function $f$ on $Z$.
The Poincar\'e series
$$P(f^k)(z)=\sum_{\gamma \in \Gamma} f^k(\gamma z)$$
defines a $\Gamma$-invariant holomorphic function for every $k \geq 2$.
If the fundamental group of the smooth locus $X^{\rm sm}$ is infinite,
then an argument of Gromov (\cite[3.2]{Gromov91}, see also \cite[Chapter13]{Kollar95} or the introduction of \cite{Takayama00}) shows that these are non-constant holomorphic functions for some $k$, which is impossible since $Z/\Gamma$ is projective.

\bigskip
Now we let $X$ be a Fano threefold with canonical singularities. Let $\mathcal{X}$ and $Z$ be the spaces constructed at the beginning of the section.
\begin{lem}
There exists a morphism $X'\to X$, such that $\mathcal{X}'=\mathcal{X}\times_X X'$ gives a terminalization of $\mathcal{X}$.
\end{lem}
\begin{proof}By the construction,  $\mathcal{X}\to X$ is isomorphic over the smooth locus of $X$. By the formula \cite[5.20]{KM98}, we know that any divisor $F$ with $a(F,\mathcal{X})=0$, indeed comes from a base change of a divisor $E$ on $X$, with
$$0\le a(E,X)=\frac{1-g}{g}\le0,$$
where $g$ is the branch degree of $\mathcal{O}_{K(E), X}\to \mathcal{O}_{K(F),\mathcal{X}}$. (This obviously implies $g=1$, but we do not need it.)

It follows from the minimal model program (\cite[1.4.3]{BCHM10}) that we can construct a model $X'\to X$ which precisely extracts all such $E$ with the corresponding $F$ satisfies $a(F,\mathcal{X})=0$. Then by the definition we know that $\mathcal{X}'=\mathcal{X}\times_X X'$ is a terminalization of $\mathcal{X}$.
\end{proof}
Let $Z'$ be the fiber product of $Z$ and $\mathcal{X}'$ over $\mathcal{X}$. Then $Z'$ is a complex analytic space with terminal singularities and the group $\Gamma$ acts on $Z'$.
By the existence of a spectral sequence relating the $L^2$-index of $Z'$ and $Z$ and the fact that $Z$ has rational singularities, we know that $\chi_{(2)}(Z',\mathcal{O}_{Z'})=\chi_{(2)}(Z,\mathcal{O}_Z)$.
Combining this with Corollary \ref{weakfano}, we have the following result.
 \begin{thm}
 Notation as above.
 $$\chi_{(2)}(Z,\mathcal{O}_Z)=-c_2(X')\cdot K_{X'}+\sum\frac{1}{24|I_i|}\sum_j (r_{ij}-\frac1{r_{ij}})>0,$$
 where the sum is taken over all the singularities $P_i$ on $\mathcal{X}'$, $I_i$ is the inertial group, $r_{ij}$ is the index of a singularity in the basket, and the second sum is took over all singularities in the associated basket of $P_i$.
 \end{thm}
 Then the same argument as in the previous section shows that the group $\Gamma$ has to be finite.

\bigskip

\end{document}